\documentclass{article}
\usepackage{mathrsfs}

\usepackage{amsthm}
\usepackage{amssymb, amsmath}
\usepackage{subfigure}
\usepackage{times}  
\usepackage{graphicx}
\usepackage{graphics}
\usepackage{fancyref}
\usepackage{float}

\theoremstyle{remark}
\newtheorem*{remark}{Remark}

\newtheorem*{definition}{Definition}
\newtheorem*{lemma}{Lemma}
\newtheorem*{Lemma}{Lemma}
\newtheorem*{theorem}{Theorem}
\newtheorem*{corollary}{Corollary}
\newtheorem*{example}{Example}
\newtheorem*{Example}{Example}

\begin{document}

\title{A brief review on geometry and spectrum of graphs}
\author{Yong Lin \and Shing-Tung Yau }
\date{April 12, 2012}
\maketitle

\setlength{\baselineskip}{1.25\baselineskip}

\setcounter{page}{1}

\section{Introduction}
This is a survey paper. We study the Ricci curvature and spectrum of graphs, as well as
the exterior forms and deRahm cohomology on graphs.

\section{Geometry and Spectral theory for Graphs}

\subsection{Ricci curvature on graphs}

Let $G=(V,E)$ be a graph, where $V$ is a vertices set and $E$ is the set of edges. For $x,y\in{V}$, $x\sim{y}$ means that $x$ is adjacent to $y$.Let $d_x$ denote the degree of the vertex $x$. If $d_x<+\infty$ for
all $x\in{V}$, we say that $G$ is a locally finite graph. If $d_x$
is same for every $x$, we say that the graph is a regular graph. For two
vertices $x$ and $y$, the distance between $x$ and $y$ is the number
of edges in the shortest path joining $x$ and $y$.The diameter of a graph $G$ is the maximum distance between any two vertices of $G$. We always assume that $G$ is connected, which means that any two vertices of
$G$ can be connected by a path in $G$.

The first definition of Ricci curvature was introduced by Fan Chung
Graham and S. T. Yau in 1996\cite{cy96}.  In the course of obtaining a good
log-Sobolev inequality, they found the following definition of Ricci
curvature to be useful:

We say that a regular graph $G$ has a local $k$-frame at a vertex $x$ if there exists injective mappings $\eta_1, \ldots, \eta_k$ from a neighborhoood of $x$ into $V$ so that
\begin{enumerate}
\item $x$ is adjacent to $\eta_ix$ for $1\leq i \leq k\,$; \\
\item $\eta_i \,x  \neq \eta_j \,x~~$ if $i \neq j\,$. \\
\end{enumerate}

The graph $G$ is said to be Ricci-flat at $x$ if there is a local $k$-frame in a neighborhood of $x$ so that for all $i\,$,
$$\bigcup_j \left(\eta_i \eta_j\right) x = \bigcup_j \left(\eta_j \eta_i\right) x~.$$

For a more general definition of Ricci curvature, we need the following.

We first define the Laplace operator on graph without loops and
multiple edges. The description in the following can be used for weighted graphs.
But for simplicity, we set all weights here equal to 1.

Let
$V^R=\{f|f:V\rightarrow{R}\}$. The Laplace operator $\Delta$ of a
graph $G$ is
$$\Delta{f(x)}=\frac{1}{d_x}\sum_{y\sim{x}}[f(y)-f(x)]$$
for all $f\in{V^R}$.  For graphs, we have
$$|\nabla{f(x)}|^2=\frac{1}{d_x}\sum_{y\sim{x}}[f(y)-f(x)]^2~.$$

We first introduce a bilinear operator
$\Gamma:V^R\times{V^R}\rightarrow{V^R}$ by
$$\Gamma(f,g)(x)=\frac{1}{2}\{\Delta(f(x)\cdot{g(x)})-f(x)\Delta{g(x)}-g(x)\Delta{f(x)}\}~.$$
The Ricci curvature operator $\Gamma_2$ is defined by iterating the
$\Gamma$:
$$\Gamma_2(f,g)(x)=\frac{1}{2}\{\Delta\Gamma(f,g)(x)-\Gamma(f,\Delta{g})(x)-\Gamma(g,\Delta{f})(x)\}~.$$
\begin{definition}
The Laplace operator $\Delta$ on graphs satisfies the
curvature-dimension type inequality $CD(m,K)$ $(m\in(1,+\infty])$ if
$$\Gamma_2(f,f)(x)\geq\frac{1}{m}(\Delta(f(x)))^2+k(x)\cdot\Gamma(f,f)(x)$$
We call $m$ the dimension of the operator $\Delta$ and $k(x)$ the
lower bound of the Ricci curvature of the operator $\Delta$.
\end{definition}

It is easy to see that for $m<m^{'}$, the operator $\Delta$
satisfies $CD(m^\prime,k)$ if it satisfies $CD(m,k)$.

\begin{remark}
We find that:
\begin{eqnarray*}
\Delta(|\nabla{f(x)}|^2)&=&\frac{1}{d_x}\sum_{y\sim{x}}\frac{1}{d_x}\sum_{z\sim{y}}[f(x)-2f(y)+f(z)]^2\\
&
&-\frac{2}{d_x}\sum_{y\sim{x}}\frac{1}{d_y}\sum_{z\sim{y}}[f(x)-2f(y)+f(z)][f(x)-f(y)]
\end{eqnarray*}

and
$$\Gamma(f,\Delta{f})(x)=\frac{1}{2}\frac{1}{d_y}\sum_{y\sim{x}}[f(y)-f(x)][\Delta{f(y)}-\Delta{f(x)}]~.$$
\end{remark}

From the definition of the Ricci curvature operator, we obtain:
\begin{eqnarray*}
\Gamma_2(f,f)(x)&=&\frac{1}{4}\frac{1}{d_x}\sum_{y\sim{x}}\frac{1}{d_y}\sum_{z\sim{y}}[f(x)-2f(y)+f(z)]^2\\
& &-\frac{1}{2}\frac{1}{d_x}\sum_{y\sim{x}}[f(y)-f(x)]^2
\\ & &+\frac{1}{2}[\frac{1}{d_x}\sum_{y\sim{x}}(f(y)-f(x))]^2~.
\end{eqnarray*}

We see that the lower bound of the Ricci curvature $k(x)$ only depends
on those vertices with distant at most 2 from vertex $x$. So we can
choose the function $f(x)$  to be supported only on those vertices.

\begin{itemize}
  \item Suppose $\Delta$ is the Laplace-Beltrami operator on a
  m-dimension complete connected Riemannian manifold $M$.

  \medskip

  Bochner's
  formula indicates that
  $$\Gamma_2(f,f)(x)=Ric(\nabla{f(x)},\nabla{f(x)})+\|Hessf(x)\|_2^2~,$$

  where $Ric$ is the Ricci tensor on $M$ and $\|Hessf\|_2$ is the
  Hilbert-Schmidt norm of the tensor of the second derivative of
  $f$. Since $\|Hessf\|_2^2\geq\frac{1}{m}(\Delta{f})^2$, so we say
  that $\Delta$ satisfies the $CD(m,k)$ if and only if the Ricci
  curvature at $x$ on $M$ is bounded below by $k$.

\item In 1985, Bakry and Emery\cite{be1985} considered these notions on metric
  measure space where the operator $\Delta$ is the so-called diffusion
  operator, that is $\Delta$ satisfies the following chain-rule
  formula: for every $C^\infty$ function $\Phi$ on $R$ and every
  function $f$,
  $$\Delta\phi(f)=\phi^{'}(f)\Delta{f}+\phi^{''}(f)\Gamma(f,f).$$
  The Laplace operator $\Delta$ on graphs does not satisfy the
  chain-rule formula. But by the following Lemma, it is the correct
  operator for defining the Ricci curvature operator on graphs.
\end{itemize}

In the classical case, $\Gamma(f,f)=\frac{1}{2}|\nabla(f)|^2$. On
   graphs, we also have

   \begin{Lemma}
   Suppose $G$ is a locally finite graph and $\Delta$ is the Laplace
   operator on $G$, then $$\Gamma(f,f)=\frac{1}{2}|\nabla{f}|^2.$$
   \end{Lemma}

 \begin{proof}
    \begin{eqnarray*}
   \Delta(f^2(x))&=&\frac{1}{d_x}\sum_{y\sim{x}}[f^2(y)-f^2(x)]\\
   &=&\frac{1}{d_x}\sum_{y\sim{x}}[f(y)-f(x)]\cdot[f(y)+f(x)]\\
   &=&\frac{2}{d_x}\sum_{y\sim{x}}[f(y)-f(x)]\cdot{f(x)}+\frac{1}{d_x}\sum_{y\sim{x}}[f(y)-f(x)]^2\\
   &=&2f(x)\Delta{f(x)}+|\nabla{f(x)}|^2.
   \end{eqnarray*}
   So $$\Gamma(f,f)=\frac{1}{2}[\Delta(f)^2-2f\Delta{f}]=\frac{1}{2}|\nabla(f)|^2.$$
   \end{proof}

 For the Ricci curvature operator on graphs, we have the following
   theorem.

   \

   \begin{theorem}[Y. Lin,  S. T. Yau\cite{ly09}]
    Suppose $G$ is a locally finite graph and
   $d=\sup_{x\in{V}}d_x$ (where $d$ can be infinite). Then we have
   $$\Gamma_2(f,f)\geq\frac{1}{2}(\Delta{f})^2+(\frac{1}{d}-1)\Gamma(f,f)$$
   i.e. the Laplace operator $\Delta$ on $G$ satisfies
   $CD(2,\frac{1}{d}-1)$.
   \end{theorem}

  \
\begin{remark}

The Ricci-flat graph in the sense of Chung and Yau satisfies $CD(\infty,0)$.

\end{remark}

We also can obtain an eigenvalue estimate for the finitely connected
   graph.

Suppose a function $f(x)\in{V^R}$ satisfies
$$(-\Delta)f(x)=\lambda{f(x)}~,$$
    then $f(x)$ is called the eigenfunction of the Laplace operator
    $\Delta$ on $G$ with eigenvalue $\lambda$.

We prove the following result which is similar to the classical result of Li and Yau on compact manifold with Ricci curvature bounded below.

\medskip

    \begin{theorem}[Y. Lin,  S. T. Yau\cite{ly09}]
    Suppose $G$ is a connected graph with diameter $D$, then the non-zero
    eigenvalue of Laplace operator $\Delta$ on $G$
    $$\lambda\geq\frac{1}{d\, D(\exp(d\,D+1)-1)}~.$$
    \end{theorem}
   \parskip=2\baselineskip

   Notice that there is a well-known estimate for the eigenvalue of Laplacian on graphs
    $$\lambda\geq\frac{1}{D\, Vol G}~,$$ where $VolG=\sum_{x\in{V}}d_x$.

The proof of the above Theorem is similar to the proof of Li and Yau on
    manifold by using the following gradient estimate
    $$\frac{|\nabla{f(x)}|^2}{(\beta-f(x))^2}\leq{d}(\frac{1}{\beta-1}\lambda+1)^2~,$$
    where $\beta>1$.

For the graph with Ricci curvature bounded below by a positive
number, There are the following eigenvalue estimate which is similar
to the Lichnerowicz theorem in Riemannian manifold. We can also show
the estimate is shape for $m=\infty$ and $m=2$.

\begin{theorem}[Fan Chung, Y. Lin, Y. Liu\cite{cll10}]
Suppose a finite graph $G$ satisfies the $CD(m,k)$ with $k>0$, then
the non-zero eigenvalue of $\Delta$ on $G$
$$\lambda\geq\frac{m}{m-1}k.$$
\end{theorem}

\section{Harnack inequality and eigenvalue estimate on graphs}

\subsection{Harnack inequality and eigenvalue estimate on graphs}

In this section, we will establish the Harnack inequality,
as a consequence, we get
eigenvalue estimate for graphs with  Ricci curvature bounded below by some constants.

The idea is to use the maximum principle similar to the Euclidien and manifold case.

We can prove the following Harnack type inequality. The idea of proof comes from \cite{cy95'}
and \cite{CHUNG-1997}.

\begin{theorem}[Fan Chung, Y. Lin,   S. T. Yau\cite{cly09}]
Suppose a finite graph $G$ satisfies the $CD(m,k)$, $f\in{V^R}$ is
an eigenfunction of Laplacian $\Delta$ with eigenvalue $\lambda$.
Then the following inequality holds for all $x\in{V}$.
$$|\nabla{f(x)}|^2\leq[(8-\frac{2}{m})\lambda-4k]\sup_{z\in{V}}f^2(x).$$
\end{theorem}

As a consequence, we can get an eigenvalue estimate for graphs.

\begin{theorem}[Fan Chung, Y. Lin,   S. T. Yau\cite{cly09}]
Suppose a finite connected graph $G$ satisfies $CD(m,k)$ and  $\lambda$
is a non-zero eigenvalue of Laplace operator $\Delta$ on $G$. Then
$$\lambda\geq\frac{1+4kdD^2}{d\cdot (8-\frac{2}{m})\cdot{D}^2}$$
where $d$ is the maximum degree and $D$ denotes the diameter of $G$.
\end{theorem}

\begin{remark}

This means we still have an eigenvalue lower bound for graphs with
Ricci curvature bounded below by some negative number, that is when
$$ k >-\frac{1}{4dD^2}.$$

\end{remark}

 For graph with non-negative Ricci curvature.
We have:
\begin{corollary}
Suppose a finitely connected graph $G$ satisfies $CD(m,0)$, then the
non-zero eigenvalue of $\Delta$ on $G$
$$\lambda\geq\frac{1}{(8-\frac{1}{m})d\,D^2}~.$$
\end{corollary}
\begin{remark}
Fan Chung and Yau proved a similar result for so-called Ricci-flat
graphs\cite{cy96}. The Ricci-flat graphs satisfy $CD(\infty,0)$.
Chung and Yau gave an example that showed their eigenvalue estimate
$\lambda\geq\frac{1}{8d\,D^2}$ is a sharp one for Ricci-flat graph. If
$m<\infty$, then we have a better estimate. For example, if $G$ is a
triangle graph with three vertices, then $m=2$.
\end{remark}

This is the key lemma to prove the
functional inequalities in this section.

\begin{lemma}\label{lem1}
Suppose $G$ is a   finite connected graph satisfying $CD(m,k)$, then
for $x\in{V}$, we have
$$(\frac{4}{m}-2)(\Delta{f})^2+(2+2k)|\nabla{f}|^2\leq\frac{1}{d_x}
\sum_{xy\in E}\frac{1}{d_y}\sum_{yz\in E}[f(x)-2f(y)+f(z)]^2.$$
\end{lemma}

 By using the above lemma, from the $CD(m,0)$ condition, we can obtain
$$\Delta(|\nabla{f(x)}|^2)\geq{2\lambda|\nabla{f(x)}|^2}-\frac{4}{m}{\lambda}^2f(x)~.$$
Meanwhile, we have
$$\Delta{f^2(x)}=2\lambda{f^2(x)}-|\nabla{f(x)}|^2~.$$
We can then use a maximum principle for function $|\nabla{f}|^2(x)+\alpha\lambda{f}^2(x)$
as Chung and Yau did.

For the proof of the corollary, we shall note that for the
eigenfunction
$f(x)$, we have ~~$\sum_{x\in{V}}d_xf(x)=0\,.$\\

\subsection{Higher order eigenvalues}

First we consider distance functions on vertices set $V$. We assume that we are given some distance function on V and denote it by $\rho_\xi(x)$.
Denote by $B_\xi(x)$ the ball defined by $\rho$, that is
\begin{displaymath}
B_\xi(r) = \{x:\rho_{\xi}(x) < r\}
\end{displaymath}
Let us assume that $\rho_{\xi}$ has the following property:
\begin{displaymath}
\mid \nabla_{xy}\rho_{\xi}\mid \leq 1
\end{displaymath}
for any edge $xy \in E$ and for any vertex $\xi \in V$, where
$$\nabla_{xy}f=f(y)-f(x)$$
is the value of the gradient of $f$ assigns on each ordered pair $x,y\in V$.

Next, we will need the following constant characterizing a structure of edges at the boundary of the ball $B_{\xi}(r)$. Given points $\xi,x\in V$, consider the following sum of $\sigma_{xy}$ over all points $y$ adjacent to $x$ and satisfying $\rho_{\xi}(y)<\rho_{\xi}(x)$:
\begin{displaymath}
\mu_x^{(\xi)} = \sum_{\{y:y\sim x \: and \: \rho_{\xi}(y)<\rho_{\xi}(x)\}} \sigma_{xy}
\end{displaymath}
Here $\sigma_{xy}$ is the weight of edge $xy$. Clearly $\mu_x^{(\xi)}\leq \mu_x$, where $\mu_{x}$ is defined to be $\sum_{y:y\sim x}\sigma_{xy}$. We regard
$\mu_{x}$ as a measure on vertices, namely for any subset $\Omega$ of vertices, $\mu(\Omega)=\sum_{x\in \Omega}\mu_x$.

We define the \textsl{spring ratio} $\nu_r$, for any $r > 0$, as follows
\begin{displaymath}
\nu_r = \inf_{\xi \in M,x\in B_{\xi}(r)}\frac{\mu_x}{\mu_x^{(\xi)}}
\end{displaymath}

Together with the function $\rho_{\xi}(x)$, we consider another function $q_\xi(x)$ - an analogue of the square distance. We postulate the following properties of q, for some positive constants $\delta,\iota$ and $R_0$:

\begin{enumerate}
\item \label{con1} $q_\xi(x) \geq 0$, and $q_\xi(x) = 0$ if and only if $x = \xi$.
\item \label{con2} For any vertex $\xi$ and for arbitrary adjacent vertices $x,y \in B_\xi(R_0)$,
    \begin{displaymath}
    \nabla_{xy}q_{\xi} \leq \rho_{\xi}(x) + \iota;
    \end{displaymath}
    clearly, we can always assume that
    \begin{displaymath}
    \iota \geq 1.
    \end{displaymath}
\item \label{con3}For any vertex $\xi$ and all $x\in B_\xi(R_0)$,
\begin{displaymath}
\Delta q_{\xi}(x)\geq \delta.
\end{displaymath}
\end{enumerate}

\begin{example}
Let $(\Gamma,\sigma)$ be the rectangular lattice graph defined on $\mathbb{Z}^n$. Let us consider
\begin{displaymath}
\rho_t(x) = \displaystyle{\max_{1\leq i \leq n}}\mid x_i - \xi_i\mid
\end{displaymath}
and
\begin{displaymath}
q_{\xi}(x)=\frac{1}{2}\sum_{i = 1}^n (x_i - \xi_i)^2
\end{displaymath}
\end{example}
In other word, $\rho_\xi(x)$ is the $l^{\infty}$-distance whereas $q_\xi(x)$ is determined by the $l^2$-distance.
It can be verified that condition \ref{con1} - condition \ref{con3} are true under these settings.

\begin{definition}
Given positive numbers $\delta,\iota$ and $R_0$, we say that a weighted graph $(\Gamma,\sigma)$ has property $P(\delta,\iota,R_0)$ if there exist a distance function $\rho$ satisfying $\nabla_{xy}\rho_\xi \leq 1$ and a function $q$ satisfying the hypotheses \ref{con1} - \ref{con3} such that
\begin{displaymath}
n:= \delta \nu_{R_0+1} \geq 1
\end{displaymath}
where $\nu_{R_0+1}$ and $\delta$ which are the spring ratio and the parameter brought by the condition \ref{con3} above.
\end{definition}

\begin{theorem}[Fan Chung, A. Grigoryan, S. T. Yau\cite{cgy00}]
Assume that the weighted graph $(\Gamma,\sigma)$ has property $P(n,\iota,R_0)$ and that
\begin{displaymath}
\omega' := \inf_{x\sim y} \sigma_{xy} > 0
\end{displaymath}
and denote
\begin{displaymath}
\omega:= \inf_{x\in V} \mu_x
\end{displaymath}
Then for any finite set $\Omega \subset V$ the Dirichlet eigenvalue $\lambda_k(\Omega)$ satisfies
\begin{displaymath}
\lambda_k(\Omega) \geq a(\frac{k}{\mu(\Omega)})^{\frac{2}{n}}
\end{displaymath}
provided
\begin{displaymath}
\mid\Omega\mid \geq k \geq \kappa \frac{\mu(\Omega)}{\omega R_0^n}
\end{displaymath}
where $\kappa = c_1(n)\iota^n > 0$ and $a = c_2(n)\iota^{-2}\nu_{R_0+1}^{-2}\omega^{'2}\omega^{2/n-2} > 0$ and $\mid \Omega\mid$ is the number of vertices.
\end{theorem}

As a corollary we have
\begin{theorem}[Fan Chung, A. Grigoryan, S. T. Yau\cite{cgy00}]
For the lattice graph on $\mathbb{Z}^n$, for any finite subset $\Omega$ of vertices, we have
\begin{displaymath}
\lambda_k(\Omega) \geq a (\frac{k}{\mu(\Omega)})^{\frac{2}{n}}
\end{displaymath}
where $k$ is any integer between $1$ and $\mid \Omega \mid$ and $a = a(n) > 0$.
\end{theorem}

\section{Modified Ollivier's Ricci curvature on graphs and Ricci-flat graphs}

\subsection{Modified Ollivier's Ricci curvature}

Recently, Lin-Lu-Yau\cite{LLY} have modified the definition of Ollivier (for
Ricci curvature of Markov chains on metric spaces\cite{ollivier}) to define Ricci
curvature of a graph in the following way.A probability distribution is a function $m: V \longrightarrow [0,1]$ so that $\sum_{x \in V} m(x) =1\,$.For any two probability distributions, $m_1$ and $m_2\,$, the transportation distance is

$$W(m_1,m_2)= \sup_f \sum_{x\in V} f(x) \left[m_1(x) - m_2(x)\right]$$

where $f$ is any Lipschitz function with constant one:

$$|f(x)-f(y)|\leq  d(x,y)$$

Given $0\leq \alpha \leq 1\,$, we define the probability distributions

$$m^{\alpha}_x(v)=\left\{
    \begin{array}{cll}
      \alpha & \mbox{ if } v = x~;\\
      \frac{1-\alpha}{d_x} & \mbox{ if } v\sim x~;\\
      0 & \mbox{ otherwise~.}
    \end{array}\right.$$

For any $x,y\in V\, $ define $\alpha$-Ricci curvature
$$
\kappa_\alpha(x,y)=1-\frac{W(m^\alpha_x, m^\alpha_y)}{d(x,y)}~.
$$

{\textbf{Lemma 1}.~} $\kappa_\alpha$ concaves upward for $0 \leq \alpha \leq 1\,$.

{\textbf{Lemma 2}.~} $\kappa_\alpha(x,y)\leq (1-\alpha)\dfrac{2}{d(x,y)}.$

When $\alpha=0$, $\kappa_0(x,y)$ is Ricci curvature defined by Ollivier. We have the following
theorem for $\kappa_0(x,y)$.

    \begin{theorem}[Y. Lin,  S. T. Yau\cite{ly09}]
    The Ricci curvature of Ollivier
$\kappa_0(x,y)\geq \frac{2}{d_{x}}+\frac{2}{d_{y}}-2$ if $d_{x}>1$ and $d_{y}>1$;
$\kappa_0(x,y)=0$ if $d_{x}=1$ or $d_{y}=1$.
    \end{theorem}
   \parskip=2\baselineskip

The lower bound can be achieved when the graph is a tree. This was showed by
 Jost and Liu\cite{Jost} recently. They also find a relation between the Ricci curvature of
 Bakry-Emery and Olliver on graphs.

We define
$$\lim_{\alpha \to 1} \frac{\kappa_\alpha(x,y)}{1-\alpha} = \kappa(x,y)$$
to be the Ricci curvature for all pairs (x,y).
From this definition, we know that
$$\kappa(x,y)\geq \kappa_0(x,y).$$

\begin{example}

\begin{enumerate}

\item The complete graph $K_n$ has a constant Ricci curvature $\frac{n}{n-1}\,$ for every edge.
\item The cycle $C_n$ for $n\geq 6$ has constant Ricci curvature $0\,$ and
$$  \kappa(C_3)=\frac{3}{2}~,~~
  \kappa(C_4)= 1~,~~
  \kappa(C_5) = \frac{1}{2}~.$$

\end{enumerate}
\end{example}

For graphs $G$ and $H$,
the Cartesian product $G\square H$ is a graph
given by $V(G)\times V(H)$ and the
two pairs $(u_1,v_1)$ and $(u_2,v_2)$ can be connected iff
$u_1=u_2$ and $v_1v_2\in E(H)$ or $u_1u_2\in E(G)$ and $v_1=v_2\,$.

\begin{theorem}[Y. Lin, L.Y. Lu, S. T. Yau\cite{LLY}]
If $G$ is $d_G$-regular and $H$ is $d_H$-regular, then

\begin{eqnarray*}
  \kappa^{G\square H}((u_1,v), (u_2,v))&=&\frac{d_G}{d_G+d_H} \kappa^G(u_1,u_2)\\
  \kappa^{G\square H}((u,v_1), (u,v_2))&=&\frac{d_H}{d_G+d_H} \kappa^H(v_1,v_2).
\end{eqnarray*}

\end{theorem}

\begin{remark}
This theorem is not true if we replace $\kappa(x,y)$ by  $\kappa_{o}(x,y)$ . This is one
of the advantage of our modified Ricci curvature.
\end{remark}

\begin{corollary}
Suppose $G$ is regular and has constant curvature $\kappa$. Then
$G^n\,$, the $n$-th power of the Cartesian product of $G$,
has constant curvature ~$\dfrac{\kappa}{n}$.
\end{corollary}

Bonnet-Myers type theorem for graphs.

\begin{theorem}[Y. Lin, L.Y. Lu, S. T. Yau\cite{LLY}]
If $\kappa(x,y)>0$, then
$$d(x,y)\leq  \left[\frac{2}{\kappa(x,y)}\right]~.$$
If for any edge ${\bar{xy}}$, $\kappa(x,y)\geq c>0$,
then
$$diam (G)\leq \frac{2}{c}\,,$$
Also, $\lambda_1 \geq c\,,$
and
\begin{eqnarray*}
Vol(G) &=& number~of~vertices~\\
&\leq& 1+ \sum_{k=1}^{\lfloor \frac{2}{c}\rfloor} \Delta^k \prod_{i=1}^{k-1}\left(1-\frac{ic}{2}\right)
\end{eqnarray*}
where $\Delta$ is the maximum degree of $G\,$.
\end{theorem}

A random graph $G(n,p)$ is a graph on $n$ vertices in which a pair of vertices appears as an edge with probability $p$.

For a random graph $G(n,p)$, we have

\begin{theorem} [Y. Lin, L.Y. Lu, S. T. Yau\cite{LLY}]
Suppose that $xy$ is an edge of random graph $G(n,p)$. The following
statements hold for the curvature $\kappa(x,y)$.
\begin{enumerate}

\item   If $p\geq \sqrt[3]{(\ln n)/n}$,  almost surely, we have
$$\kappa(x,y)=p+O\left(\sqrt{\frac{\ln n}{np}}\right).$$
In particular, if $p\gg \sqrt[3]{(\ln n)/n}$, almost surely,
we have $\kappa(x,y)=(1+o(1))p$.

\item  If $\sqrt[3]{(\ln n)/n}> p \geq 2\sqrt{(\ln n)/n}$,
almost surely, we have
$$\kappa(x,y)=O\left(\frac{\ln n}{np^2}\right).$$
\item  If $1/\sqrt{n}\gg p\gg \sqrt[3]{(\ln n)/n^2}$,
 almost surely, we have
$$\kappa(x,y)=-1+O(np^2)+O(\frac{\ln n}{n^2p^3}).$$
\item  If $\sqrt[3]{1/n^2}\gg p\gg \frac{\ln n}{n}$,
 almost surely, we have
$$\kappa(x,y)=-2+O(n^2p^3)+ O\left( \sqrt{\frac{\ln n}{np}} \right).$$
\end{enumerate}
\end{theorem}

Note that we say that a property $P$ is almost surely satisfied if
the limit of the probability that $P$ holds goes to $1$ as $n$ goes
to infinity.

\subsection{Ricci-flat graphs}

Ricci-flat manifolds are Riemannian manifolds with Ricci curvature vanishes. In Physics, they represent vacuum solutions to the analogues of Einstein's equations for Riemannian Manifolds of any dimension, with vanishing cosmological constance. The important class of Ricci-flat manifolds is Calabi-Yau manifolds. This follows from Yau's proof of the Calabi conjecture, which implies that a compact $K\ddot{a}hler$ manifold with a vanishing first real Chern class has a $K\ddot{a}hler$ metric in the same class with vanishing Ricci curvature. There are many works to find the Calabi-Yau manifolds. Yau conjectured that there are finitely many topological types of compact Calabi-Yau manifolds in each dimension. This conjecture is still open. In this paper, we will use our modified Ollivier's Ricci curvature study this question on graphs.

We recall the definition of Ricci curvature on graphs introduced by Fan Chung and Yau in 1996\cite{cy96}.
We say that a regular graph $G$ has a local $k$-frame at a vertex $x$ if there exist injective mappings $\eta_1,\dots,\eta_k$ from a neighborhood of $x$ into $V$ so that
\begin{enumerate}
\item $x$ is adjacent to $\eta_ix$ for $1\leq i \leq k$;
\item $\eta_i x \ne \eta_j x$ if $i\ne j$
\end{enumerate}
The graph $G$ is said to be Ricci-flat at x if there is a local $k$-frame in a neighborhood of x so that for all i,
\begin{displaymath}
\cup_j (\eta_i\eta_j)x = \cup_j(\eta_j\eta_i)x.
\end{displaymath}

It is easy to show that the Ricci flat graphs defined by Chung and Yau are the Ricci-flat graphs in the sense of Ollivier's definition, and are the graphs with non-negative Ricci curvature of our modified definition.

The girth of a graph is the length of a shortest cycle contained in the graph.

The following theorem is our main result:

\begin{theorem} [Y. Lin, L.Y. Lu, S. T. Yau\cite{RF}]
Suppose that $G$ is a Ricci flat graph with girth $g(G) \geq 5$, then $G$ is one of the following graphs:
\begin{enumerate}
\item the infinite line
\item cycle $C_n$ with $n\geq 6$
\item the dodecahedral graph
\item the Petersen graph
\item the half-dodecahedral graph
\end{enumerate}
\end{theorem}

\begin{figure}
\centering
\subfigure[Dodecahedral]{\includegraphics[width=1.3in]{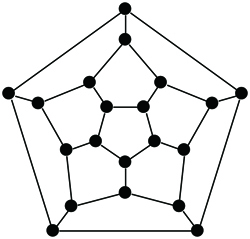}}
\subfigure[Petersen]{\includegraphics[width=1.3in]{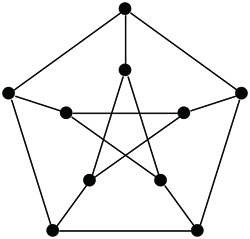}}
\subfigure[Half-dodecahedral]{\includegraphics[width=1.3in]{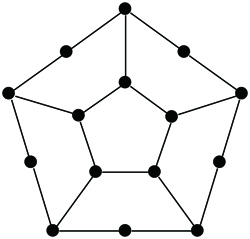}}

\end{figure}

Since the Cartesian product of two regular Ricci-flat graphs is still Ricci-flat, there are infinite number of Ricci-flat graphs with girth 4.
Same happens for girth 3.

The homogeneous graph $\Gamma$ associated with an abelian group $\mathcal{H}$ is Ricci-flat in the sense of Chung and Yau. That means that $\mathcal{H}$ is a subgroup of automorphism group of $\Gamma$ and $\mathcal{H}$ acts transitively on the vertex set $V$ of $\Gamma$, i.e. for any two vertex $u$ and $v$ there is an $f \in \mathcal{H}$ such that $f(u) = v$. So there are tremendous number of Ricci-flat graphs in the sense of Chung and Yau and therefore Ricci-flat graphs in the sense of Ollivier. They are either girth 4 or 2. Our modified Ricci curvature are good definition to classify the Ricci-flat graphs.

\section{Exterior forms on digraphs(A. Grigor'yan, Y. Lin, Y. Muranov, S. T. Yau\cite{DF})}

\subsection{homology and cohomology of digraphs}

 A differential calculus on an associative algebra A is an algebraic analogue of the calculus of differential forms on a smooth manifold. The discrete differential calculus is based on the universal differential calculus on an associative algebra of functions on discrete set. By a natural way, we can consider this calculus as a calculus on a universal digraph with the given discrete set of vertexes. This approach gives an opportunity to define differential calculus for every subgraph of the universal digraph.

Given a finite set $V$, we define a $p$-form $\omega$ on $V$ as $\mathbb{K}$-valued function on $V^{p+1}$. The set of all $p$-forms is a linear space over $\mathbb{K}$ that is denoted by $\Lambda^p$. It has a canonical basis $e^{i_0\dots i_p}$. For any $\omega \in \Lambda^p$ we have
\begin{displaymath}
\omega = \sum_{i_0,\dots,i_p\in V}\omega_{i_0\dots i_p}e^{i_0\dots \i_p}
\end{displaymath}
where $\omega_{i_0\dots i_p} = \omega(i_0,\dots,i_p)$. The exterior derivative $d$:$\Lambda^p \rightarrow \Lambda^{p+1}$ is defined by
\begin{displaymath}
(d\omega)_{i_0\dots i_{p+1}} = \sum^{p+1}_{q=0}(-1)^q \omega_{i_0\dots \hat{i_q}\dots i_{p+1}}
\end{displaymath}
and satisfies $d^2 = 0$, where the hat $\hat{i_q}$ means omission of the index $i_q$.

We define a subspace $\mathcal{R}^p \subset \Lambda^p$ of regular forms that is spanned by $e_{i_0\dots\i_p}$ with regular paths $i_0\dots i_p$ (when$i_k \ne i_{l+1}$), and observed that the spaces $\mathcal{R}^p$ are invariant for $d$.

A $p$-path on $V$ is a formal linear combination of the elementary $p$-paths $e_{i_0\dots i_p}\equiv i_0\dots i_p$ and the linear space of all $p$-paths is denoted by $\Lambda_p$. For any $v\in \Lambda_p$ we have
\begin{displaymath}
v = \sum_{i_0,\dots,i_p\in V}v^{i_0\dots i_p}e_{i_0\dots i_p}
\end{displaymath}
and a pairing with a $p$-path $\omega$
\begin{displaymath}
(\omega,v) = \sum_{i_0,\dots,i_p}\omega_{i_0\dots i_p}v^{i_0\dots i_p}
\end{displaymath}
The dual operator $\partial: \Lambda_{p+1} \rightarrow \Lambda_p$ is given by
\begin{displaymath}
\partial e_{i_0\dots i_{p+1}} = \sum_{q=0}^{p+1} (-1)^q e_{i_0\dots \hat{i_q}\dots i_{p+1}}
\end{displaymath}

Let $I_p$ be the subspace of $\Lambda_p$ that is spanned by $e_{i_0\dots i_p}$ with irregular paths $i_0\dots i_p$. Then the spaces $I_p$ are invariant for $\partial$, which allows to define $\partial$ on the quotient spaces $\mathcal{R} = \Lambda_p / I_p$. For simplicity of notation we identify the elements of $\mathcal{R}_p$ with their representatives that are regular $p$-paths. Then $e_{i_0\dots i_p}$ with irregular $i_0\dots i_p$ are treated as zeros.

A digraph is a pair of $(V,E)$ where $V$ is an arbitrary set and $E$ is a subset of $V\times V \backslash \text{diag}$. The elements of $V$ are called \textit{vertices} and the elements of $E$ are called(\textit{directed}) \textit{edges}. The set $V$ will be always assumed non-empty and finite.

Let $i_0\dots i_p$ be an elementary regular $p$-path on V. It is called \textsl{allowed} if $i_ki_{k+1}\in E$ for any $k = 0,\dots,p-1$, and \textsl{non-allowed} otherwise. The set of all allowed elementary $p$-paths will be denoted by $E_p$,and non-allowed by $N_p$. For example $E_0 = V$ and $E_1 = E$.

Denote by $\mathcal{A}_p = \mathcal{A}_p(V,E)$ the subspace of $\mathcal{R}_p$ spanned by the allowed elementary $p$-paths, that is,
\begin{displaymath}
\begin{split}
\mathcal{A}_p &= span\{e_{i_0\dots i_p}:i_0\dots i_p \in E_p\} \\
&= \{ v\in \mathcal{R}_p : v^{i_0\dots i_p} = 0 \: \forall i_0\dots i_p \in N_p\}
\end{split}
\end{displaymath}
The elements of $\mathcal{A}_p$ are called allowed $p$-paths.

Similarly, denote by $\mathcal{N}^p$ the subspace of $\mathcal{R}^p$, spanned by the non-allowed elementary $p$-forms, that is,
\begin{displaymath}
\begin{split}
\mathcal{N}^p &= span\{e^{i_0\dots i_p}: i_0\dots i_p \in N_p\} \\
&= \{ \omega \in \mathcal{R}^p : \omega_{i_0\dots i_p} = 0 \: \forall i_0\dots i_p \in E_p\}
\end{split}
\end{displaymath}
Clearly, we have $\mathcal{A}_p = (\mathcal{N}^p)^{\perp}$ where $\perp$ refers to the annihilator subspace with respect to the couple ($\mathcal{R}^p,\mathcal{R}_p$) of dual spaces.

We would like to reduce the space $\mathcal{R}^p$ of regular $p$-forms so that the non-allowed forms can be treated as zeros. Consider the following subspaces of spaces $\mathcal{R}^p$
\begin{displaymath}
\boxed
{\mathcal{J}^p\equiv \mathcal{J}^p(V,E):= \mathcal{N}^p + d\mathcal{N}^{p-1}}
\end{displaymath}
that are $d$-invariant, and define the space $\Omega^p$ of $p$-forms on the digraph $(V,E)$ by
\begin{displaymath}
\boxed{\Omega^p \equiv \Omega^p(V,E):=\mathcal{R}^p/\mathcal{J}^p}
\end{displaymath}

Then $d$ is well-defined on $\Omega^p$ and we obtain a cochain complex
\begin{displaymath}
0 \rightarrow \Omega^0 \overset{d}{\longrightarrow} \dots \overset{d}{\longrightarrow} \Omega^p \overset{d}{\longrightarrow} \Omega^{p+1} \overset{d}{\longrightarrow} \dots
\end{displaymath}

Shortly we write $\Omega = \mathcal{R}/\mathcal{J}$ where $\Omega$ is the complex and $\mathcal{R}$ and $\mathcal{J}$ refer to the corresponding cochain complexes.

If the digraph $(V,E)$ is complete, that is, $E=V\times V \setminus diag$ then the spaces $\mathcal{N}^p$ and $\mathcal{J}^p$ are trivial, and $\Omega^p = \mathcal{R}^p$.

Consider the following subspaces of $\mathcal{A}_p$
\begin{displaymath}
\boxed
{\Omega_p \equiv \Omega_p(V,E) = \{v\in \mathcal{A}_p:\partial v \in \mathcal{A}_{p-1}\}}
\end{displaymath}
that are $\partial$-invariant. Indeed, $v\in \Omega_p \Rightarrow \partial v \in \mathcal{A}_{p-1} \subset \Omega_{p-1}$. The elements of $\Omega_p$ are called $\partial$-invariant $p$-paths.

We obtain a chain complex $\Omega$
\begin{displaymath}
0 \longleftarrow \Omega_0 \overset{\partial}{\longleftarrow} \Omega_1\overset{\partial}{\longleftarrow} \dots\overset{\partial}{\longleftarrow} \Omega_{p-1}\overset{\partial}{\longleftarrow} \Omega_{p} \overset{\partial}{\longleftarrow}\dots
\end{displaymath}
that, in fact, is dual to $\Omega$.

By construction we have $\Omega_0 = \mathcal{A}_0$ and $\Omega_1 = \mathcal{A}_1$ so that
\begin{displaymath}
dim \Omega_0 = \mid V\mid \quad and \quad dim\Omega_1 = \mid E\mid,
\end{displaymath}
while in general $\Omega_p \subset \mathcal{A}_p$

Let us define the (co)homologies of the digraph $(V,E)$ by
\begin{displaymath}
H^p(V,E) := H^p(\Omega)\quad and\quad H_p(V,E) := H_p(\Omega)
\end{displaymath}
Recall that $H^p(V,E)$ and $H_p(V,E)$ are dual and hence their dimensions are the same. The values of $dim H_p(V,E)$ can be regarded as invariants of the digraph $(V,E)$.

Note that for any $p\geq 0$
\begin{displaymath}
dim H_p(\Omega) = dim \: \Omega_p - dim \: \partial \Omega_p - dim\partial \Omega_{p+1}
\end{displaymath}
Let us define the Euler characteristic of the digraph $(V,E)$ by
\begin{displaymath}
\chi(V,E) = \sum^{n}_{p=0}(-1)^p dim H_p(\Omega)
\end{displaymath}

provided $n$ is so big that
\begin{displaymath}
dim H_p(\Omega) = 0 \quad \forall \: p > n
\end{displaymath}

\begin{Example}
Consider the graph of 6 vertices $V = \{0,1,2,3,5\}$ with 8 edges $E = \{01,02,13,14,23,24,53,54\}$.

\begin{figure}
\centering
\includegraphics[width=1.2in]{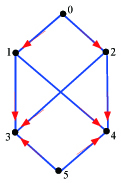}
\caption{1}
\label{fig1}
\end{figure}
\end{Example}

Let us compute the spaces $\Omega_p$ and the homologies $H_p(\Omega)$. We have
\begin{displaymath}
\begin{split}
&\Omega_0  =  \mathcal{A}_0 = span\{e_0,e_1,e_2,e_3,e_4,e_5\},\quad dim\Omega_0 = 6\\
&\Omega_1  =  \mathcal{A}_1 = span\{e_{01},e_{02},e_{13},e_{14},e_{23},e_{24},e_{53},e_{54}\},\quad dim\Omega_1 = 8\\
&\mathcal{A}_2 = span\{e_{013},e_{014},e_{023},e_{024}\},\quad dim\mathcal{A}_2 = 4
\end{split}
\end{displaymath}

The set of semi-edges(we say a pair $ij$ of vertices is a \textsl{semi-edge} if it is not an edge but there is a vertex $k$ so that $ik$ and $kj$ are edges) is $\mathcal{S} = \{e_{03},e_{04}\}$ so that $dim\Omega_2 = dim\mathcal{A}_2 - \mid S\mid =2$. The basis in $\Omega_2$ can be easily spotted as each of two squares $0,1,2,3$ and $0,1,2,4$ determine a $\partial$-invariant 2-path, whence
\begin{displaymath}
\Omega_2 = span\{e_{013} - e_{023},e_{014} - e_{024}\}
\end{displaymath}

Since there are no allowed 3-paths, we see that $\mathcal{A}_3 = \Omega_3 = \{0\}$. It follows that
\begin{displaymath}
\chi = dim \Omega_0 - dim \Omega_1 + dim \Omega_2 = 6 - 8 + 2 =0
\end{displaymath}

Let us compute dim $H_2$:
\begin{displaymath}
dim H_2 = dim\Omega_2 - dim\partial\Omega_2 - dim\partial\Omega_3 = 2 - dim\partial\Omega_2.
\end{displaymath}
The image $\partial\Omega_2$ is spanned by two 1-paths
\begin{displaymath}
\begin{split}
\partial(e_{013}-e_{023}) &= e_{13} - e_{03} + e_{01} - (e_{23} - e_{03} +e_{02})\\ &= e_{13} + e_{01} - e_{23} - e_{02}\\
\partial(e_{014}-e_{024}) &= e_{14} - e_{04} + e_{01} - (e_{24} - e_{04} +e_{02})\\ &= e_{14} + e_{01} - e_{24} - e_{02}
\end{split}
\end{displaymath}
that are clearly linearly independent. Hence, $dim \partial\Omega_2 = 2$ whence $dim H_2 = 0$.

 The dimension of $H_1$ can be computed similarly, but we can do easier using the Euler characteristic:
\begin{displaymath}
dim H_0 - dim H_1 + dim H_2 = \chi = 0
\end{displaymath}
whence dim $H_1 = 1$.

In fact, a non-trivial element of $H_1$ is determined by $1$-path
\begin{displaymath}
v = e_{13} = e_{14} -e_{53} +e_{54}
\end{displaymath}
Indeed, by a direct computation $\partial v = 0$, so that $v\in ker\partial\mid_{\Omega_1}$while for $v$ to be in $Im\partial\mid_{\Omega_2}$ it should be a linear combination of $\partial(e_{013}-e_{023})$ and $\partial(e_{014}-e_{024})$,which is not possible since they do not have the term $e_{54}$.

We can do some transformations of digraphs to get the homology and cohomology group for new graphs.
\

We can define the Hodge Laplacian on $p$-forms, and study the spectrum of such Laplacian. There are interesting properties of these spectrum,
such as the torsion defined by them, that can be developed parallel to Riemannian geometry.

Naturally these spectrum are invariants of the graph that can give a great deal of information about the graph.

\subsection{Minimal paths and hole detection}

The elements of $H_p(V,E)$ can be regarded as $p$-dimensional holes in the digraph $(V,E)$. To make this notion more geometric, we can work with representatives of the homologies classes, which are closed $p$-paths. We say that two closed $p$-paths $u$ and $v$ are homological and write $u\sim v$ if $u$ and $v$ represent the same homology class, that is, if $u\sim v$ is exact.

For any $p$-forms $v$ define its length by \begin{displaymath}
l (v) = \sum_{i_0,\dots,i_p\in V} \mid v^{i_0\dots i_p}\mid
\end{displaymath}

Given a closed $p$-path $v_0$, consider the minimization problem
\begin{displaymath}
l(v) \mapsto min \quad for \quad v\sim v_0
\end{displaymath}
This problem always has a solution, although not necessarily unique. Any solution is called a \textsl{minimal} $p$-path. It is hoped that minimal $p$-paths(in a given homology class) match our geometric intuition of what holes in a graph should be. We can give some examples of minimal paths to support this claim. The following is one of the example in dimension 2.

\begin{example}
Consider a digraph on Figure 2\ref{fig2}.
\begin{figure}[H]

\centering
\includegraphics[width=1.5in]{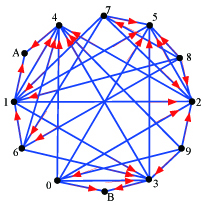}
\caption{2}
\label{fig2}
\end{figure}

Removing successively the vertices $A,B,8,9,6,7$ we obtain a digraph $(V',E')$ with $V' = \{0,1,2,3,4,5\} $ and $E'=\{02,03,04,05,12,13,14,15,24,25,34,35\}$ that has the same homologies as $(V,E)$. The digraph $(V',E')$ is shown in two ways on the following figure. Clearly the second representation of this graph is reminiscent of an octahedron.
\begin{figure}[H]
\centering
\includegraphics[width=3.0in]{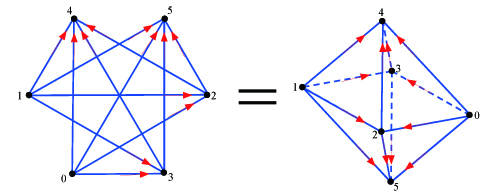}
\caption{3}
\label{fig3}
\end{figure}

The digraph $(V',E')$ is the same as the 2-dimensional sphere-graph. Hence we obtain that $dim H_2(V,E) = 1$ while $H_p(V,E) = {0}$ for $p=1$ and $p> 2$.

Consider a closed 2-path on $(V,E)$
\begin{displaymath}
\begin{split}
v_0 &= e_{024} - e_{025} -e_{034} +e_{035} -e_{124} +e_{125}\\
&+e_{134} - e_{135} -e_{634} +e_{614} -e_{613}
\end{split}
\end{displaymath}
Then a solution to the minimization problem is given by
\begin{displaymath}
v = e_{024} - e_{025} -e_{034} +e_{035} - e_{124} +e_{125}+e_{134}-e_{135}
\end{displaymath}
that is a 2-path that determines a 2-dimensional hole in $(V,E)$ given by the octahedron. Note that on Figure 2 \ref{fig2} this octahedron is hardly visible, but it can be computed purely algebraically as shown above.
\end{example}

\bibliographystyle{amsalpha}

\begin{thebibliography}{A}


\bibitem{be1985} D. Bakry and M. Emery, \textit{Diffusions
hypercontractives},
S\'eminaire de probabilit\'es, XIX, 1983/84, 177--206, Lecture Notes in Math. 1123, Springer, Berlin, 1985.




\bibitem{CHUNG-1997}
Fan R. K., Chung, \textit{Spectral Graph Theory}, CBMS
Regional Conference Series in Mathematics, 1997, Number 92, American
Mathematical Society.








\bibitem{cy95'}
Fan Chung and S.-T. Yau,
\textit{A Harnack inequality for homogeneous graphs and subgraphs},
Comm. Anal. Geom.
 2 (1994), 627--640,
also in Turkish J. Math. \textbf{19} (1995), 273--290.






\bibitem{cll10}
Fan Chung, Yong Lin and Yuan Liu, \textit{Curvature Aspects of Graphs}, preprint.

\bibitem{cly09}
Fan Chung, Yong Lin and S.-T. Yau, \textit{Harnack inequalities on graphs  with Ricci curvature bounded below
 }, preprint.

\bibitem{cy96}
Fan Chung and S.-T. Yau,
 \textit{Logarithmic Harnack inequalities},
Math. Res. Lett. \textbf{3} (1996), 793--812.










\bibitem{cgy00}
Fan Chung, A. Grigor'yan and S.-T. Yau,
  \textit{Higher eigenvalues and isoperimetric inequalities on Riemannian
  manifolds and graphs},
 Comm. Anal. Geom. \textbf{8}
 (2000), 969--1026.

\bibitem{DF}
A. Grigor'yan, Y. Lin and Y. Muranov and S.-T. Yau, \textit{Differential forms on digraphs}, preprint.

\bibitem{Jost}
J. Jost and S. P. Liu, \textit{Ollivier¡¯S Ricci Curvature, Local clustering and
curvature dimension inequalities on graphs}, preprint.


\bibitem{LY-1979} P. Li and S. T. Yau, \textit{Estimates of eigenvalues of a
compact Riemanian manifold}, AMS Symposium on the Geometry of the
Laplace Operator, University of Hawaii at Manoa, 1979, 205-239.

\bibitem{LLY}Y. Lin , L.Y. Lu and S. T. Yau, \textit{Ricci Curvature of
graphs},  Tohoku Mathematical Journal,Vol.63(605-627),2011.

\bibitem{RF}Y. Lin , L.Y. Lu and S. T. Yau, \textit{Ricci-flat graphs with girth at least five}, preprint.


\bibitem{ly09}
Yong Lin and S.-T. Yau, \textit{Ricci curvature and eigenvalue estimate on
locally finite graphs}, Mathematical Research Letters \textbf{17}(2010), 345--358.





\bibitem{ollivier}
Y. Ollivier,
\textit{Ricci curvature of Markov chains on metric spaces},
J. Funct. Anal. \textbf{256} (3) (2009), 810--864.




\end{thebibliography}

\bigskip

\end{document}